\RequirePackage{fix-cm}
\documentclass[smallextended]{svjour3}       % onecolumn (second format)
\smartqed  % flush right qed marks, e.g. at end of proof
\usepackage{graphicx}
\usepackage{amsmath, amsfonts, amssymb}
\usepackage{mathtools}
\usepackage{enumitem}
\begin{document}

\title{Unbounded $\sigma$-order-to-norm continuous and $un$-continuous operators%\thanks{Grants or other notes
}


\author{Mina Matin               \and
        Kazem Haghnejad Azar$^*$ \and
        Razi Alavizadeh}


\institute{M. Matin \at
           Department  of  Mathematics  and  Applications, Faculty of Sciences, University of Mohaghegh Ardabili, Ardabil, Iran. \\
           \email{minamatin1368@yahoo.com}                 %  \\
%          \emph{Present address:} of F. Author  %  if needed
           \and
           K. Haghnejad Azar \at
           Department  of  Mathematics  and  Applications, Faculty of Sciences, University of Mohaghegh Ardabili, Ardabil, Iran. \\
           \email{haghnejad@uma.ac.ir}
           \and
           R. Alavizadeh \at
           Department  of  Mathematics  and  Applications, Faculty of Sciences, University of Mohaghegh Ardabili, Ardabil, Iran. \\
           \email{ralavizadeh@uma.ac.ir}
}

\date{Received: date / Accepted: date}

\maketitle

\begin{abstract}
An operator $T $ from a vector lattice $E$ into a normed lattice
$F$ is called unbounded $\sigma$-order-to-norm continuous
whenever $x_{n}\xrightarrow{uo}0$ implies
$\| Tx_{n}\|\rightarrow 0$,
for each sequence $(x_{n})_n\subseteq E$. 
For a net $(x_{\alpha})_{\alpha}\subseteq E$,
if $x_{\alpha}\xrightarrow{un}0$ implies
$Tx_{\alpha}\xrightarrow{un}0$,
then  $T$ is called an unbounded  norm continuous operator.
 In this manuscript, we study some
properties of these classes of operators and their relationships
with the other classes of operators.
\end{abstract}
%%%%%%%%%%%%%%%%%%%%%%%%%%%%%%%%%%%%%%%%%%
\section{Introduction and Preliminaries}
Let $E$ be a vector lattice and $F$ be a normed lattice. In the second section of  this 
manuscript, we will study and investigate on operators $T:E\rightarrow F$ which carry 
every unbounded order convergent sequence into norm convergent sequence. 
The collection of all unbounded 
$\sigma$-order-to-norm continuous operators will be denoted by
$ L_{uon}^\sigma(E,F)$. We show that under some conditions the modulus of an unbounded $\sigma$-order-to-norm continuous operator exists and belongs to $ L_{uon}^\sigma(E,F)$ and every $\sigma$-$uon$-continuous operator $T$ is a Dunford-Pettis operator.
In the third section, we will introduce a new classification of operators named as $un$-continuous operators and we will investigate on some properties of them and their relationships with other classifications of operators.

To state our results, we need to fix some notations and recall some definitions.
Throughout this paper, $E$ is a vector lattice,   the subset $ E^+=\{x\in E:x\geq 0\}$  is called the positive cone of $E$ and the elements of $E^+$ are called the positive elements of $E$. A subset $A\subseteq E$ is called order bounded if there exists $a,b\in E$ such that $A\subseteq [a,b]$ where $[a,b]=\{x\in E:~a\leq x\leq b\}$. An operator $T:E\rightarrow F$ between two vector lattices is said to be order bounded if it maps order bounded subsets of $E$ to order bounded subsets of $F$.
If $E$ is a normed space, then $ E^{\prime} $ is the topological dual space of $ E $ and
$ T^{\prime}:F'\to E' $ is the  adjoint of operator $T:E\to F$ between two normed space.
%%%%%%%%%%%%%%%%%%%%%%%%%%%%%%%%%%%%%
A sequence $ (x_n)_n$ in a vector lattice $E$ is 
said to be disjoint whenever $n\neq m$ implies $x_n \perp x_m$.
%%%%%%%%%%
If $A$ is a 
nonempty   subset of vector lattice $E$, then its disjoint 
complement $ A^d$ is defined by $A^d = \{x\in E:\ x \perp y$  for 
all $y \in A\}$.
%%%%%%%%%%%%%%%%%%%
%%%%%%%%%%%%%%%%%%%%%
An order closed ideal  of  $E$ is referred to as a band. A band $B$ in 
a vector lattice $E$ that satisfies $ E = B \oplus B^d$ is 
referred to as a projection band. Let $B$ be a projection band in 
a vector lattice $E$. Thus every vector $x\in E$ has a unique 
decomposition $x = x_1 + x_2$, where $x_1 \in B$ and $x_2 \in B^d$. Then it 
is easy to see that a projection $P_B : E \rightarrow E $ is defined via 
the formula $ P_B (x ) = x_1$. 
Clearly, $P_B $ is a positive projection. Any projection of the 
form $P_B$ is called a band projection.
%%%%%%%%%%%%%%%%%%%%%%
%%%%%%%%%%%%%%%%%%%%%%%%%%%%%%%

Let $E$ be a vector lattice and $x\in E$.
A net $(x_{\alpha})_{\alpha \in A}\subset E$ is said to be:
\begin{itemize}
\item
\textbf{order convergent} to $x$ if there is a net
$(z_{\beta})_{\beta \in B} $ in $ E $ such that
$ z_{\beta} \downarrow 0 $ and for every $ \beta \in B$,
there exists $\alpha_{0} \in A$ such that
$ | x_{\alpha} - x |\leq z_{\beta}$ whenever $ \alpha \geq \alpha_{0}$.
We denote this convergence by $ x_{\alpha} \xrightarrow{o} x $ and write 
that $ (x_{\alpha})_{\alpha} $ is $o$-convergent to $x$.
%%%%%%%%%%%%%%%%
\item
\textbf{unbounded order convergent} to $x$ if
$ | x_{\alpha} - x | \wedge u \xrightarrow{o} 0$ for all
$ u \in E^{+} $. We denote this convergence by
$ x_{\alpha} \xrightarrow{uo}x $ and write that
$ (x_{\alpha})_{\alpha} $ is $uo$-convergent to $x$.
It was first introduced by Nakano in \cite{14}
and was later used by DeMarr in \cite{4}. 
%%%%%%%%%%%%%%%%
\item
\textbf{unbounded norm
convergent} to $x$ if $\| | x_{\alpha} - x | \wedge u\|\rightarrow 0$ for 
all $ u \in E^{+} $.  We denote this convergence by
$ x_{\alpha} \xrightarrow{un}x $ and write that
$ (x_{\alpha})_{\alpha} $ is $un$-convergent to $x$.
It was studied in \cite{5,12}.
\end{itemize}
%%%%%%%%%%%%%%%%%%%%%%%%%%%%%%%%
It is clear that for order bounded 
nets, $uo$-convergence is equivalent to $o$-convergence.
%%%%%%%%%%%%
By Corollary 3.6 of \cite{8}, every disjoint sequence in vector 
lattice $E$ is $uo$-null.
In \cite{15}, Wickstead characterized the spaces in 
which weak convergence of nets implies $uo$-convergence and vice 
versa and in \cite{7}, Gao characterized the spaces $E$ such 
that in its dual space $E^\prime$, $uo$-convergence implies 
weak$^{*}$-convergence and vice versa. 
%%%
\\
Let $X$ and $Y$ be two Banach spaces and let $E$ and $F$
be two Banach lattices. An operator,
\begin{itemize}

%%%%%%%%%%%%%%%%%%%%%%
\item%
$ T: X \rightarrow Y $, is said to be
\textbf{Dunford-Pettis} (or that $ T $ has the Dunford-Pettis property) 
whenever $ x_{n}\xrightarrow{w} 0$ in $ X $ implies
$\| Tx_{n}\|\rightarrow 0$.
%%%%%%%%%%%%%%%%%%%%%%
\item%
$ T : E \rightarrow X $,  is said to be
\textbf{$M$-weakly compact} if $T$ is continuous and $ \lim \|Tx_{n}\| = 0$ holds for
every norm bounded disjoint sequence $(x_{n})_n $ of $ E $. 
%%%%%%%%%%%%%%%%%%%%%%
\item%
 $ T:E \rightarrow F $, is said to be \textbf{disjointness preserving} 
whenever $ x \perp y $ in $ E $ implies $ Tx \perp Ty$ in $F$.
\end{itemize}
%%%%%%%%%%%%%%%%%%%%
If for an operator $T:E\rightarrow F$ between two vector 
lattices, $T\vee (-T)$ exists we say its modulus $|T|$ exists.
$e \in E^+ $ is a strong unit when the ideal $ I_{e}$ (generated by $e$) is 
equal to $ E $. Equivalently, for every $ x \geq 0 $ there exists
$ n \in \mathbb{N}$ such that  $ x \leq ne$. $e \in E^+ $ is also  a quasi-interior point if the closure of $I_e$ equal with $E$; or equivalently, $x\wedge ne\xrightarrow{\Vert.\Vert} x$ for every $x\in E^+$.
%%%%%%%%%%%%%%%
A non-zero element $ a\in E^+$ is an atom, if $ x \perp y  $ and 
$ x,y \in [0,a]$ imply either $ x = 0 $ or $ y = 0$.
$E$
is called an atomic Banach lattice if it is the
band generated by its atoms. 
%%%%%%%%%%%%%%%%%%%
A Banach lattice $E$ is said to be $KB$-space whenever every 
increasing norm bounded sequence of $E^+$ is norm convergent. Recall that $E^\sim$ is the vector space of all order bounded linear 
functionals on $E$ and $E_n^\sim $ is the vector space of all order 
continuous linear functional on $E$. A sublattice $Y$ of a vector lattice 
$E$ is said to be regular if for every subset $A$ of $Y$, 
%%%%%%%%%%
infimum of $A$ is the same in $E$ and in $Y$, whenever infimum of $A$ 
exists in $Y$. It is easy to see that $L_p[0,1]$ is an order dense regular 
sublattice of $L_0[0,1]$, whenever $0<p<\infty $.  
A vector lattice is called laterally complete whenever every subset of 
pairwise disjoint positive vectors has a supremum.
For unexplained notation the reader is referred to \cite{1}.

\section{Unbounded $\sigma$-order-to-norm continuous operators}
Let $E$ be a vector lattice  and $F$ be a normed lattice.
An operator $T: E \rightarrow F $ is said to be unbounded
$\sigma$-order-to-norm continuous (or, $\sigma$-$uon$-continuous for short),
if $ x_{n} \xrightarrow{uo} 0$ in $ E $ implies
$ Tx_{n} \xrightarrow{\|.\|} 0 $ in $F$. The collection of all unbounded 
$\sigma$-order-to-norm continuous operators will be denoted by
$ L_{uon}^\sigma(E,F) $.

%%%%%%%%%%%%%%%%%%%
In the following example for $0\leq p\leq +\infty$, we show that every lattice homomorphism $T:L_p[0,1]\rightarrow \mathbb{R}$ is $\sigma$-$uon$-continuous operator.
%%%%%%%%%%%
\begin{example}
Let $ 0\leq p < \infty$ and let $T: L_p[0,1] \rightarrow \mathbb{R}$ be a 
lattice homomorphism, then $T$ is a $\sigma$-$uon$-continuous operator. 
First note that since $T$ is positive, $T$ is order bounded. Since 
$L_p[0,1]^\sim = L_p[0,1]^\sim_n $, $T$ is order continuous.
If $p=0$, by Exercise 23 of page 214 from \cite{1a},
$L_p[0,1]$ is Dedekind complete and laterally complete.
Therefore, by Theorem 3.2 of \cite{12a}, $T$ is $\sigma$-$uon$-continuous.
Now let $0< p < \infty$, since $T$ is
order continuous and lattice homomorphism, 
by Theorem 2.32 of \cite{1}, there exists an extension
$S$ of $T$ from $L_0[0,1]$ to $\mathbb{R}$,
which is an order continuous lattice homomorphism.
If $(x_n)_n \subset L_p[0,1]$ and $x_n \xrightarrow{uo}0 $ in $L_p[0,1]$, 
by Theorem 3.2 of \cite{8}, $x_n \xrightarrow{uo}0$ in $L_0[0,1]$ and 
therefore by Theorem 3.2 of \cite{12a}, $x_n \xrightarrow{o}0$ in
$L_0[0,1]$. Thus, $ Tx_n = Sx_n\xrightarrow{o}0$ in $\mathbb{R}$.
Hence $Tx_n \xrightarrow{\|.\|}0$ in $\mathbb{R}$.
\end{example}
%%%%%%%%%%%%%%%%
Let $E$ be a vector lattice and $G$ and $F$ be two normed lattices.
Then, it is obvious that for each  $\sigma$-$uon$-continuous operator
$ T :E \rightarrow G $  and continuous operator $ S:G \rightarrow F $, 
$ ST :E \rightarrow F $ is a $\sigma$-$uon$-continuous operator.
The following proposition shows that with some conditions,
the combination of two $\sigma$-$uon$-continuous operators
is $\sigma$-$uon$-continuous.
\begin{proposition} 
Let $E$ be a vector lattice and $G$ and $F$ be two Banach lattices.
If $ E \xrightarrow{T} G \xrightarrow{S} F $ are $\sigma$-$uon$-continuous 
operators and the linear span of minimal ideals in $ G $ is order dense in 
$ G$, then $ ST $ is likewise a $\sigma$-$uon$-continuous operator.
\end{proposition}
\begin{proof}
Let $ (x_{n})_n \subseteq E $ and $ x_{n} \xrightarrow{uo} 0 $ in $ E$. 
Then by the assumption we have $ Tx_{n}\xrightarrow{\|.\|}0 $ in $ G$, and 
so $ Tx_{n} \xrightarrow{w} 0 $ in $ G$. Now by Theorem 1 of \cite{15}, we 
have $ Tx_{n} \xrightarrow{uo} 0 $ in $ G$. Therefore,
$ ST(x_{n}) \xrightarrow{\|.\|} 0$ in $ F$. 
\end{proof}
Recall that an operator $T$ from vector lattice $E$ into normed vector 
lattice $F$ is said to be $\sigma$-order-to-norm continuous operator 
whenever $x_n \xrightarrow{o} 0$ implies $Tx_n \xrightarrow{\|.\|} 0$ for 
all sequence $(x_n)_n \subseteq E$. This classification of operators has been introduce and studied by K. Haghnejad Azar, see \cite{10}.
Obviously, every $\sigma$-$uon$-continuous operator is a
$\sigma$-order-to-norm continuous operator.
However, the converse is not true in general,
as shown in the following example.
\begin{example}
The operator $T:\ell^{1} \rightarrow \ell^{\infty} $ defined by 
\[
T(x_{1},x_{2},\ldots) = (\sum_{i=1}^\infty x_i , \sum_{i=1}^\infty x_{i},\ldots)
\]
is a $\sigma$-order-to-norm continuous operator
($\ell^{1}$ has order continuous norm and $ T $ is a continuous operator). 
Now  if $(e_{n})_n $ is the standard basis of $ \ell^{1} $, then
$ e_{n}\xrightarrow{uo} 0 $ in $\ell^{1}$ and $T(e_{n})=(1,1,1,\ldots)$. 
Therefore, $ \|T(e_{n})\| \nrightarrow 0 $ in $\ell^{\infty}$. Thus $ T $ 
is not $\sigma$-$uon$-continuous.
\end{example}
If a vector lattice $E$ is Dedekind $\sigma$-complete and laterally
$\sigma$-complete, then every $\sigma$-order-to-norm continuous operator 
from $E$ into $F$ is $\sigma$-$uon$-continuous.
%%%%%%%%%%%
Namely, if $(x_{n})_n$ is a sequence in $E$ such that
$ x_{n} \xrightarrow{uo} 0 $, then by Theorem 3.2 of \cite{12a},
$(x_n)_n$ is order bounded and therefore $ x_n \xrightarrow{o} 0 $ in
$ E $ and so  $ T(x_n) \xrightarrow{\|.\|} 0 $ in $ F $.
Hence $ T $ is a $\sigma$-$uon$-continuous operator.

Let $T:E\rightarrow F $ be a positive operator between two vector lattices. 
We say that an operator $ S:E\rightarrow F$ is dominated by
$T$ (or that $T$ dominates $S$) whenever $|Sx|\leq T|x|$ holds
for each $x\in E$.
Let  $T$ be a $\sigma$-$uon$-continuous operator from vector lattice $E$ 
into normed lattice $F$. It is obvious that  $S$ 
is $\sigma$-$uon$-continuous whenever  $S$ is dominated by $T$.
%%%%%%%%%%%%%%%%%%%%%%%%%%%%%%%%
%%%%%%%%%% ?????????
%In the following theorem, we  study the modulus of $\sigma$-$uon$-continuous operators.
\begin{theorem}
Let $E$ be a vector lattice and $F$ be a normed lattice.
If $ T:E \rightarrow F$ is an order bounded
$\sigma$-$uon$-continuous operator, then by one of
the following conditions, $ |T| $ exists and
belongs to $ L_{uon}^\sigma(E,F)$.
\begin{enumerate}[label=$(\arabic*)$]
	\item $E$ is Dedekind $\sigma$-complete and laterally
	$\sigma$-complete and $F$ is atomic with order continuous norm.
	\item $ T $ preserves disjointness.
\end{enumerate}
\end{theorem}
\begin{proof}
\begin{enumerate}[label=$(\arabic*)$]
\item
 First we show that $ T $ is a $\sigma$-order continuous operator.
Let  $(x_{n})_{n} \subseteq E$  and  $x_{n} \xrightarrow{o} 0$ in $ E$. 
Then $ x_{n} \xrightarrow{uo} 0$ and by Theorem 3.2 of \cite{12a},
$(x_{n})_{n}$  is order bounded in $ E$. Hence by the assumption,
$(T(x_{n}))_{n} $ is order bounded and  $T(x_{n}) \xrightarrow{\|.\|}0$
in $F$. Now by Lemma 5.1 of \cite{5},  $ T(x_{n}) \xrightarrow{o} 0 $
in $ F$. Hence $ T $ is a $\sigma$-order continuous operator.
Note that by Theorem 4.10  of \cite{1}, $ F $ is Dedekind complete, 
therefore by using Theorem 1.56 of \cite{1}, $ |T| $ is a
$\sigma$-order continuous operator. Now, assume that
$(x_{n})_n\subseteq E$ and $ x_{n} \xrightarrow{uo} 0$ in $E$,
since $E$ is Dedekind $\sigma$-complete and laterally $\sigma$-complete,
$(x_{n})_n$ is order bounded. It follows that $ x_{n} \xrightarrow{o} 0$
in $ E $ and $ |T|(x_n) \xrightarrow{o} 0 $ in $ F $. Since $ F $ has
order continuous norm, we have $ |T|(x_{n}) \xrightarrow{\|.\|} 0$.
\item It is obvious that $ F $ is Archimedean. By Theorem 2.40 of \cite{1}, 
$ |T| $ exists and for all $x$, we have $|T| (|x|) =  |T(|x|)| = |T(x)|$.	
If $(x_{n})_{n} \subseteq E$ and $x_{n} \xrightarrow{uo} 0$,
then for each $n$,
\(
|T|(|x_{n}|)=|T(|x_{n}|)|=|T(x_{n})| \xrightarrow{\|.\|} 0
\)
in $ F$. Now by inequality $ | | T | (x_{n}) | \leq |T | | x_{n} |$,
we have  $ | T| (x_{n}) \xrightarrow{\|.\|} 0$.
\end{enumerate}	
\end{proof}
Recall that a vector lattice $E$ is said to be perfect whenever
the natural embedding $x \mapsto \hat{x}$ from $E$ to
$(E_n ^\sim)_n ^\sim$  is one-to-one and onto.
By Exercise 3 of page 74 of \cite{1}, if $F$ is a perfect vector lattice, 
then $L_b (E,F)$ is likewise a perfect vector lattice for
each vector lattice $E$.	
\begin{corollary}
If the condition $(1)$ from above theorem holds, then 
\begin{enumerate}[label=$(\arabic*)$]
\item  an order bounded operator $ T : E \rightarrow F $ is a
$\sigma$-order continuous operator if and only if it is
$\sigma$-$uon$-continuous. It follows that
$ L_{uon}^\sigma (E,F) \cap L_b(E,F)  $ is a band of $L_b(E,F)$. 
\item  if $F$ is a perfect vector lattice, then 
$ L_{uon}^\sigma (E,F) \cap L_b(E,F)  $ is likewise a
perfect vector lattice for each vector lattice $E$.
\end{enumerate}
\end{corollary} 
\begin{proof}
\begin{enumerate}[label=$(\arabic*)$]
\item	 Follows from Theorem 1.57 of \cite{1}.
\item	  We will show that if $ E$ is a perfect vector lattice and $ B$ is 
a band of $E$, then $B$ is perfect vector lattice in its own right. Suppose 
$x,y \in B$ and $x\neq y$, by Theorem 1.71 of \cite{1}, there exists
$f\in E_n ^\sim $ such that $f(x)\neq f(y)$. The restrection of $f$ to
$B$ is order continuous and $f\mid_B (x) \neq f\mid_B (y)$. Therefore
$B_n ^\sim $ separates the points of $B$.
Let $(x_\alpha)_\alpha \subseteq B$, $0\leq x_\alpha \uparrow $ and
$\sup\{f(x_\alpha)\}<\infty$ for each $0 \leq f \in B_n ^\sim$. It is clear 
that $(x_\alpha)_\alpha \subseteq E$ and $0\leq x_\alpha \uparrow  $ in
$ E$. On the other hand, for each $f\in E_n ^\sim$,
$f\mid_B \in B_n ^\sim$ and $f\mid_B (x_\alpha) = f(x_\alpha) $ for all
$\alpha$, therefore $\sup\{|f(x_\alpha)|\}< \infty$ for each
$0\leq f \in E_n ^\sim$. Thus by Theorem 1.71 of \cite{1}, there exists 
some $x \in E$ satisfying $0\leq x_\alpha \uparrow x$. Since $B$ is a band 
of $E$, hence $x\in B$. Therefore by Theorem 1.71 of \cite{1},
$B$ is a perfect vector lattice. Similarly
$ L_{uon}^\sigma (E,F) \cap L_b(E,F)  $ is a perfect vector lattice.
\end{enumerate}	
\end{proof}
Recall that a Banach space $E$ is said to be $AL$-space, if
$ \|x+y\| = \|x\| + \|y\| $ holds for all $ x , y \in E^+$ with
$ x\wedge y =0$.
\begin{theorem}
Let $ E $ be an $AL$-space, $F$ be a normed lattice and let 
$ T:E\rightarrow F$ be a positive operator.
Then for the following assertions:
\begin{enumerate}[label=$(\arabic*)$]
	\item  $T$ is a $\sigma$-$uon$-continuous operator.
	\item  T is a Dunford-Pettis operator.
	\item  for every relatively weakly compact net
	$(x_{\alpha})_{\alpha} \subseteq E $,
	      $ x_{\alpha} \xrightarrow{w}0$ implies
                  	$ Tx_{\alpha} \xrightarrow{\|.\|}0$.
	\item  For every relatively weakly compact net
	$(x_{\alpha})_{\alpha} \subseteq E $,
	      $x_{\alpha} \xrightarrow{uo}0$ implies
	                $Tx_{\alpha} \xrightarrow{\|.\|}0$.
\end{enumerate}
We have
	\begin{equation*}
	(1)\Rightarrow (2) \Rightarrow (3)\Rightarrow (4).
	\end{equation*}
\end{theorem}
\begin{proof}
\begin{enumerate}[label=($\arabic*$)]
\item[] $(1)\Rightarrow(2)$	Let $(x_{n})_{n} \subseteq E $ be a norm bounded 
and disjoint sequence. Since $x_{n} \xrightarrow{uo} 0 $ in $ E $, 
therefore $Tx_{n} \xrightarrow{\|.\|}0 $ in $F$. Hence $ T $ is a
$M$-weakly compact operator. Now by Theorem 5.61 of \cite{1},
$ T $ is weakly compact and therefore by Theorems 5.85 and 5.82 of 
\cite{1}, $ T $ is a Dunford-Pettis operator.
\item[] $(2)\Rightarrow (3)$ Suppose $(x_{\alpha})_{\alpha}$ is relatively 
weakly compact and $x_{\alpha} \xrightarrow{w}0$. Suppose also
$(Tx_{\alpha})_\alpha$ does not converge to 0 in norm. Then there exists
$\epsilon > 0$ such that for any $\alpha$, there exists
$\beta(\alpha) \geq \alpha $ satisfying
$\|Tx_{\beta(\alpha)} \| \geq \epsilon$. Thus by passing to the subnet
$(Tx_{\beta(\alpha)})$, we may assume $\inf{\|Tx_{\alpha}\|}> 0$. Since
$0 \in \overline{(x_{\alpha})}^{w} $, By Theorem 4.50 of \cite{6} there exists a sequence $(y_{n})_n \subseteq \{x_{\alpha} : \alpha \}$
such that $y_{n} \xrightarrow{w} 0$. The assumption $(2)$ now implies
$ Ty_{n} \xrightarrow{\|.\|} 0$, which is a contradiction.
\item[] $(3)\Rightarrow (4)$ Suppose $(x_{\alpha})_{\alpha}$ is relatively 
weakly compact and $ x_{\alpha} \xrightarrow{uo} 0 $.
By Proposition 3.9 of \cite{9}, $| x_{\alpha}| \xrightarrow{w} 0$.
Thus by assumption (3), $ T | x_{\alpha} | \xrightarrow{\|.\|}0$.
Now by inequality $ |Tx_{\alpha}| \leq T |x_{\alpha}|$ we have
$Tx_{\alpha} \xrightarrow{\|.\|}0$.
\end{enumerate}	
	\end{proof}
\section{Unbounded norm continuous operators}
An operator $ T $ between two Banach lattices $ E $ and $ F $ is said
to be unbounded norm continuous (or, $un$-continuous for short) whenever
$ x_\alpha\xrightarrow{un} 0$ implies $Tx_{\alpha}\xrightarrow{un}0$,
for $(x_{\alpha})_{\alpha}\subseteq E$. For sequence
$(x_{n})_{n}\subseteq E$, if $x_{n}\xrightarrow{un}0$ implies
$Tx_{n}\xrightarrow{un}0$, then  $T$ is called a $\sigma$-unbounded norm 
continuous operator (or, $\sigma$-$un$-continuous for short).
The collection of all unbounded norm continuous operators of $ L(E,F) $ 
 will be
denoted by $L_{un}(E,F)$. That is,
\begin{equation*}
L_{un}(E,F)=\{T\in L(E,F):~T~\text{is unbounded norm continuous }\}.
\end{equation*} 
Similarly, $ L_{un}^\sigma(E,F) $  will denote the collection of all 
operators from $ E $ to $ F $ that are $\sigma$-unbounded norm continuous. 
That is,
\begin{equation*}
	L_{un}^\sigma(E,F)=\{T\in L(E,F):~T~\text{is $\sigma$-unbounded norm continuous }\}.
\end{equation*}
%Recall that a subset $A$ of $E$ is almost order bounded, if for every $\epsilon>0$ there exists $y\in E^+$  such that  $A\subseteq [-y,y]+\epsilon B_E$. Equivalently, $\Vert(\vert x\vert-y)^+\Vert<\epsilon$ for all $y\in A$.\\
Now in the following proposition, by using
% Lemma 2.9,
Theorem 5.3 from \cite{5} and Theorem 2.3 from \cite{12}, we show the relationship  between the classifications of  $\sigma$-unbounded norm continuous and unbounded $\sigma$-order-to-norm continuous operators. 

\begin{proposition}
Let $E$ be a Banach lattice with order continuous norm and $F$ be a Banach lattice. Then we have the following assertions:
\begin{enumerate}[label=$(\arabic*)$]
     \item if $E$ is atomic, then $L_{uon}^\sigma (E,F) \subseteq  L_{un}^\sigma(E,F)$.
     \item if $F$ has a strong unit, then $L_{un}^\sigma(E,F)\subseteq  L_{uon}^\sigma(E,F)$.
\end{enumerate}
\end{proposition}

By using Theorem 4.3 of \cite{12}, we also have the following proposition.
\begin{proposition}
Let $E$ and $F$ be two Banach lattices and	Let $ G $ be a sublattice of
$ E $ and $ T \in L_{un}(E,F)$. Each of the following conditions implies 
that $ T \in L_{un}(G,F)$.
\begin{enumerate}[label=$(\arabic*)$]
	\item $ G $ is majorizing in $ E $;
	\item $ G $ is norm dense in $ E $;
	\item $ G $ is a projection band in $ E $.
\end{enumerate}	
\end{proposition}
%%%%%%%%%%%
The preceding proposition implies that if $ T\in L_{un}(E^\delta , F)$, 
then $ T \in L_{un}(E,F)$ whenever $E$ and $F$ are Banach lattices and
$E^\delta$ is a Dedekind completion of $E$.
Deng, Brien and Troitsky in \cite{5}, show that $un$-convergent is topological. For each $\epsilon>0$ and $x\in E^+$ the collections 
$V_{\epsilon, x}=\{y\in E:~\Vert    \vert y\vert\wedge x\Vert<\epsilon\}$ is a base of zero neighbourhoods for a topology, and convergence in this topology agrees with $un$-convergence. We will refer to this topology as $un$-topology. 
An operator $T:E\rightarrow F$ between two Banach lattices is
$un$-continuous if and only if for each subset $A$ of $F$ that is
$un$-open (resp, close), $T^{-1}(A)$ is $un$-open (resp, close) in $E$.
Let $ E , G$ and $ F$ be Banach lattices.
If $ E \xrightarrow{T} G \xrightarrow{S} F $ are $un$-continuous operators, 
clearly $ ST $ is likewise a $un$-continuous operator. Also, by using 
Theorem 2.40 of \cite{1}, if $ T : E \rightarrow F $ preserves disjointness 
and $un$-continuous operator, then $ | T | $ exists and is a
$un$-continuous operator. Now in the following, we give some examples of 
$un$-continuous operators.
\begin{example}
\begin{enumerate}[label=$(\arabic*)$]
\item	Let $B$ be a projection band of Banach lattice $E$ and $ P_{B}$
the corresponding band projection. It follows easily from
$ 0 \leq P_B \leq I $(see Theorem 1.44 of \cite{1}) that if
$ x_\alpha \xrightarrow{un} 0 $ in $ E $ then
$P_B x_\alpha \xrightarrow{un} 0$ in $ B $. Therefore $ P_B $ is a
$un$-continuous operator.
\item	Let $E$ and $F$ be two Banach lattices such that $E$ has
a strong unit. Then, by Theorem 2.3 of \cite{12}, each continuous operator
$ T: E\rightarrow F $ is $un$-continuous.
\end{enumerate}
\end{example}
Recall that a topological space is said to be sequentially compact if every 
sequence has a convergent subsequence. An operator $ T : E \rightarrow F $ 
between two Banach  lattices is said to be (sequentially) $un$-compact if
$ TB_{E} $ ($B_{E}$  is closed unit ball of $E$) is relatively (sequentially) $un$-compact in $ E $. Equivalently, for every bounded net
$ (x_{\alpha})_{\alpha} $ (respectively, every bounded sequence
$(x_{n})_{n}$) its image has a subnet (respectively, subsequence), which is 
$un$-convergent. A net $(x_\alpha)_\alpha$ is $un$-Cauchy if for every
$un$-neighborhood $U$ of zero there exists $\alpha_0$ such that
$x_\alpha - x_\beta \in U$ whenever $\alpha , \beta \geq \alpha_0$.
The order continuous Banach lattice $ E$ is $un$-complete if each
$un$-Cauchy net $(x_\alpha)_\alpha $ of $E$ is $un$-convergent to $x\in E$. 
These concepts have been introduce by Kandic,  Marabeh and Troitsky, see \cite{12}.

Clearly, every  compact operator is both $un$-compact and sequentially
$un$-compact. In general  $un$-compact and  sequentially $un$-compact 
operators are not $un$-continuous and vice versa as shown
by the following example.
%%%%%%%%%%%%%%%%
\begin{example}\label{ex:no:continuous}
\begin{enumerate}[label=$(\arabic*)$]
\item The operator $ T: \ell^1 \rightarrow \ell^\infty $ defined by 
\begin{equation*}
T(x_1,x_2,\ldots) = (\sum_{i=1}^\infty x_i, \sum_{i=1}^\infty x_i,\ldots),
\end{equation*}
is clearly rank one, and so $ T $ is a compact operator. It follows that
$ T $ is $un$-compact and sequentially $un$-compact.  
If $(e_{n})_{n}$ is the standard basis of $\ell^1$, by Proposition 3.5 of 
\cite{12},  $ e_{n} \xrightarrow{un} 0$ in $\ell^1$. We have
$ T(e_{n}) = (1,1,1,...)$, therefore $ (T(e_{n}))_{n} $ is not
$un$-convergent to 0. Hence $ T $ is $un$-compact but  is not a
$un$-continuous operator.
\item Let $E=L_1[0,1]$. Clearly, the identity operator
$ I: E  \rightarrow E $ is $un$-continuous. Since $ E $ is a $KB$-space,
by Theorem 6.4 of \cite{12}, $ B_{E} $ is $un$-complete.
But since $ E $ is not atomic, by Theorem 7.5 of \cite{12},
$ B_{E} $ is not $un$-compact. Hence $ I $ is not $un$-compact.
\end{enumerate}
\end{example}
A subset $A$ of Banach lattice $E$ is said to be $un$-bounded, if $A$ is 
bounded with respect to $un$-topology. An operator $T:E \rightarrow F$ between 
two Banach lattices is $un$-bounded if $T(A)$ is $un$-bounded in $F$ for 
each $un$-bounded subset $A$ of $E$.

The proof of the following lemma is similar to Theorem 1.15 (b) from \cite{14a}.
\begin{lemma}\label{set:compact:bounded}
Each $un$-compact subset of Banach lattice $E$ is $un$-bounded.
\end{lemma}
%\begin{proof}
%	Let $U\subset E$ be a $un$-neighborhood of zero. Since scalar multiplication is $un$-continuous, there is a $\delta> 0$ and there is a $un$-neighborhood $V$ of zero in $E$ such that $\alpha V \subset U$ whenever $|\alpha| < \delta$. Let $W$ be the union of all these sets $\alpha V$. Then $W$ is a  $un$-neighborhood of zero, $W$ is balanced, and $W \subset U$. Since norm topology in $E$ stronger than $un$-topology, we have $W$ belongs to norm topology. By Theorem 1.15 of \cite{14a}, $ E = \cup _{n=1} ^\infty n W$. Now we have $A\subset \cup_{n=1}^\infty n W $. Since $A$ is a $un$-compact, there are integers $n_1 < ... < n_s$ such that  $A\subset n_1W \cup n_2 W \cup ... \cup n_s W = n_s W$. The equality holds because $W$ is balanced. Therefore $A$ is $un$-bounded.
%\end{proof}
%%%%%%%%%%%%%%%%
%\begin{corollary}\label{operator:compact:bounded}
%Each $un$-compact operator $T:E \rightarrow F$ between Banach lattices is $un$-bounded.
%\end{corollary}
Also, we need the following lemma that is a specific case of Theorem 1.30 of \cite{14a}.
\begin{lemma}\label{set:bounded:equivalent}
	Let $E$ be a Banach lattice and $A\subset E$. The following are equivalent:
	\begin{enumerate}[label=$(\arabic*)$]
	\item $A$ is $un$-bounded.
    \item If $(x_n)_n \subset A$ and $({\alpha}_n)_n $
    is a sequence of scalars such that $\alpha_n \rightarrow 0$ as
    $n\rightarrow\infty$, then ${\alpha}_n x_n \xrightarrow{un}0$ as
    $n\rightarrow\infty$.
	\end{enumerate}
\end{lemma}
%\begin{proof}
%    Suppose $A$ is a $un$-bounded. Let $V$ be a balanced $un$-neighborhood of zero in $E$. Then $A \subset tV$ for some $t$. If $x_n \in A$ and ${\alpha}_n \rightarrow 0$, there exists $N$ such that $|{\alpha}_n |t < 1$ if $n > N$. Since $t^{-1} A \subset V$ and $V$ is balanced, ${\alpha}_n x_n \in V$ for all $n > N$. Thus ${\alpha}_n x_n \xrightarrow{un} 0$.\\
%	Conversely, if $A$ is not $un$-bounded, there is a $un$-neighborhood $V$ of zero and a sequence $r_n \rightarrow \infty $ such that no $r_n V$ contains $A$. Choose $x_n \in A$ such that $ x_n \notin r_n V$. Then no $r_n^{-1}x_n$ is in $V$, so that $(r_n^{-1}x_n)$ does not $un$-converge to zero.
%\end{proof}
%%%%%%%%%%%%%%%%%%%%%%%%%%
\begin{proposition}\label{operator:properties}
	Let  $T:E \rightarrow F$ be a $un$-continuous operator between two Banach lattices. Then we have the following assertions.
\begin{enumerate}[label=$(\arabic*)$]
\item $T$ is a $un$-bounded operator.
\item If $A\subseteq E$ is a $un$-compact set, then $T(A)$ is a $un$-compact set in $F$.
\end{enumerate}	
\end{proposition}
\begin{proof}
\begin{enumerate}[label=$(\arabic*)$]
\item Let $A$ be a $un$-bounded subset of $E$.
%%%%%%%
By Lemma \ref{set:bounded:equivalent}
it is enough to show that for 
each $(Tx_n)_n \subset T(A) $ and $({\alpha}_n)_n \subset \mathbb{R} $ 
such that ${\alpha}_n \rightarrow 0$ as $n \rightarrow \infty$ we have 
${\alpha}_nTx_n \xrightarrow{un} 0$ as $n \rightarrow \infty$.
Let $(x_n)_n \subset A $ and $({\alpha}_n)_n \subset \mathbb{R} $ such that
${\alpha}_n \rightarrow 0$ as $n \rightarrow \infty$.
It follows from $un$-boundedness of $A$ and Lemma 
\ref{set:bounded:equivalent} that
${\alpha}_n x_n \xrightarrow{un}0$ as $n\rightarrow \infty$.
Since $T$ is $un$-continuous we have
$ {\alpha}_n T(x_n) = T({\alpha}_nx_n)  \xrightarrow{un}0$.
Therefore, the proof is complete.	
\item Let  $(y_\alpha)_\alpha$ be a net in $T(A)$. Then there exists net
	$(x_\alpha)_\alpha \subseteq A$   such that 
	$T(x_\alpha) = y_\alpha$ for all $\alpha$.
It follows from $un$-compactness of $A$ that there exists a subnet
$(x_{\alpha_\beta}) $  of  $(x_\alpha)_\alpha$ with
$x_{\alpha_\beta}\xrightarrow{un}x$ in $E$.
Thus, $T(x_{\alpha_\beta}) \xrightarrow{un}Tx$ in $F$.
Hence $T(A)$ is a $un$-compact set.
\end{enumerate}
\end{proof}
\begin{proposition}\label{p:3.7} 
Assume that  $E$ and $F$ are two Banach lattices and $E$ has quasi-interior point.  $T:E\rightarrow F$ is a $un$-bounded operator if and only if $T$ is $\sigma$-$un$-continuous.
\end{proposition}
\begin{proof}
It follows from Proposition \ref{operator:properties} that $T$ is  a $un$-bounded operator whenever $T$ is a $\sigma$-$un$-continuous.

Conversely, let $(x_n)_n\subseteq E$ and $x_n\xrightarrow{un} 0$. 
By Theorem 3.2 from  \cite{12},   $E$ is metrizable, and so by Theorem 1.28(b) of \cite{14a}, there exists a sequence $(\alpha_n) \subset \mathbb{R}^+$ such that
$\alpha_n\rightarrow +\infty$ and $\alpha_n x_n \xrightarrow{un}0$.
% for some positive real number $\alpha_n$ whenever .
Obviously,   $(\alpha_n x_n )$ is $un$-bounded. Therefore $(T(\alpha_n x_n ))$ is $un$-bounded. By Lemma \ref{set:bounded:equivalent} we have $Tx_n=\frac{1}{\alpha_n}T( \alpha_n x_n)\xrightarrow{un} 0$. Therefore, $T$ is $\sigma$-$un$-continuous.
\end{proof}
%%%%%%%%%%
There is a $un$-compact operator which is not $un$-bounded. Let $T:\ell^1\rightarrow \ell^\infty$  be the $un$-compact operator in  Example \ref {ex:no:continuous}. But, since $\ell^1$ has quasi-interior point and $ T $ is not a $\sigma$-$un$-continuous operator,  by Proposition \ref{p:3.7}, $T$ is not $un$-bounded.
%%%%%%%%%
%%%%%%%%%%%%%%%%%%%%%%%%%%
%%%%%%%%%%%%%%%%%%%%%%%%%%

%%%%%%%%%%%%%%%%%
\begin{proposition} Assume that  $E$ and $F$ are two Banach lattices.
\begin{enumerate}[label=$(\arabic*)$]
\item Let $ F $ has a strong unit and let $ T:E \rightarrow F $ be a 
sequentially $un$-compact  operator. Then $ T^\prime $ is both sequentially 
$un$-compact and $un$-compact.
\item If $ T^\prime : F^\prime \rightarrow E^\prime $ is a sequentially
$un$-compact  operator and $ E ^\prime $ has a strong unit, then $ T $ is 
both sequentially $un$-compact and  $un$-compact.
\end{enumerate}
\end{proposition}
\begin{proof}
\begin{enumerate}[label=$(\arabic*)$]
\item	Let $ (x_{n})_{n} $ be a bounded sequence in $ E $. Then by the 
assumption $ (T(x_{n}))_{n} $ has a subsequence, which is $un$-convergent. 
Now by Theorem 2.3 of \cite{12}, $ (T(x_{n}))_{n} $ is norm-convergent and 
therefore $ T $ is a compact operator. Now by Theorem 5.2   of \cite{1},
$ T^\prime $ is  compact and therefore it is both sequentially
$un$-compact and $un$-compact.
\item Proof is similar to (1).
\end{enumerate}	
\end{proof}

%%%%%%%%%%%%%%%%%%%%%%%%%%%%%%
%Now by using Theorem 5.3 of \cite{5} and Theorem 2.3 of \cite{12}
%we have the following result.
%%%%%%%%%%

%%%%%%%%%%%%%%%%%%%%
Recall that, for every ideal $ I $ of a vector lattice $ E $, the vector 
space $ \dfrac{E}{I} $ is a vector lattice (see page 99 from \cite{1}).   
\begin{theorem}
	Let $ T : E \rightarrow F $ be an  operator between two 
	Banach lattices  and $T(E^+) = F^+ $. If $\ker(T)$ is an ideal of $ E $, 
	then $ T $ is $un$-continuous.
\end{theorem}
\begin{proof}
At first it is clear that $T$ is surjective.
On the other hand, as kernel of $ T $ is an ideal of $ E $, by Theorem 2.22 of 
\cite{1}, the quotient vector space $ \frac{E}{\ker(T)}$ is a 
vector lattice and the operator $ S : E \rightarrow \frac{E}{\ker(T)}$ 
defined by $ S(x) = x+\ker(T) $ is a Riesz homomorphism. 
Now we define the operator $ K : \frac{E}{\ker(T)} \rightarrow F $ 
via $ K(x+\ker(T)) = Tx$. It is clear that $ K $ is well defined, 
one-to-one operator and $ T = KS$. Let $y\in F $    
and  $x\in E $ such that  $ y = T(x) = K(x+\ker(T))$. It follows that $ K $ is a 
surjective operator. Now if  $ x+\ker(T) \geq 0$,
%_\frac{E}{\ker(T)}$, 
then $ |x+\ker(T)| = x+\ker(T)$. On the other hand, since $ S $ is 
Riesz homomorphism, by Theorem 2.14 of \cite{1},
we have $ S(|x|) = |S(x)|$. Therefore $ |x|+\ker(T) = |x+\ker(T)| = x+\ker(T) $
and hence $ K(x+\ker(T)) = K(|x|+\ker(T)) = T(|x|) \geq 0$. Thus $ K $ is 
a positive operator. Let $ y \in F^+ $. Since $ T(E^+) = F^+ $, 
it follows that there exists $ x \in E^+ $ that $ y = T(x) = K(x+\ker(T))$.
It follows from $ K^{-1}(y) = K^{-1}(K(x+\ker(T))) = x+\ker(T) $ and
$|x+\ker(T)| = |x|+\ker(T) = x+\ker(T)$ that $ K^{-1}$ is a 
positive operator. Now by Theorem 2.15 of \cite{1}, $ K $ is a 
Riesz homomorphism.
It follows that $T$ is
%%%%%%%%%%
Riesz homomorphism. Now, the proof is complete by this fact that
each surjective Riesz homomorphism is $un$-continuous.
\end{proof}
Let $E$ be a vector lattice and $E^{\sim \sim}$ be the bidual of $E$. 
Recall that a subset $A$ of $E$ is $b$-order bounded in $E$ 
if $A$ is order bounded in $E^{\sim\sim}$.
\begin{theorem}
	Let $E$ be a Banach lattice with order continuous norm and $F$ be a 
	Dedekind $\sigma$-complete Banach lattice such that the norm of $F$ is 
	not order continuous. If each operator $T:E\rightarrow F$ is
	$un$-continuous, then $E$ is $KB$-space.
\end{theorem}
\begin{proof}
	By way of contradiction,  suppose that $E$ is not a $KB$-space.
	Then by  Lemmas 2.1 and 3.4 of \cite{2}, there exists a $b$-order 
	bounded disjoint sequence $(x_n)_n$ of $ E^+$ such that $\|x_n\|=1$ for 
	all $n$, and there exists a positive disjoint $(g_n)_n$ of $E^\prime$ 
	with $\|g_n\| \leq 1$ such that $g_n (x_n)=1 $ for all $n$,
	and $g_n(x_m)=0$ for $n\neq m$.
	Now we consider the operator $ S:E\rightarrow \ell^\infty $ defined by 
	$ S(x) = (g_k(x))_{k=1}^\infty$ for all $x \in E$.
	Since $(x_n)_n$ is disjoint,  by Corollary 3.6 of \cite{8},  it is
	$uo$-null. Hence $x_n \xrightarrow{uo} 0$ in $ E$ and therefore
	$x_n \xrightarrow{un} 0$ in $E$. It follows that
	$\|S(x_n)\| = \|(g_k(x_n))_{k=1}^{\infty}\| \geq g_n (x_n)=1$.
	By Theorem 2.3 of \cite{12} $S(x_n)\nrightarrow 0$ with respect to
	$un$-topology in $\ell^\infty$. Therefore $S$ is not $un$-continuous.
	On the other hand, since the norm of $F$ is not order continuous and
	$F$ is Dedekind $\sigma$-complete, by Corollary 2.4.3 of \cite{13},
	$ F$ contains a complemented copy of $\ell^\infty$. 
	Now we consider the composed operator
	$T= I\circ S :E \rightarrow \ell^\infty \rightarrow F$, where $I$
	is the canonical injection of $\ell^\infty$ into $F$. This operator
	is not $un$-continuous which is impossible, and so the proof follows.
\end{proof}
%%%%%%%%%%%%%%%%%%%%%%%%%%%%%
%\begin{acknowledgements}
%\end{acknowledgements}


%% BibTeX users please use one of
%%\bibliographystyle{spbasic}      % basic style, author-year citations
%%\bibliographystyle{spmpsci}      % mathematics and physical sciences
%%\bibliographystyle{spphys}       % APS-like style for physics
%%\bibliography{}   % name your BibTeX data base
%
%% Non-BibTeX users please use
%\begin{thebibliography}{}
%%
%% and use \bibitem to create references. Consult the Instructions
%% for authors for reference list style.
%\end{thebibliography}

\end{document}